\documentclass[12pt]{amsart} 

\title[Conformal dimension bounds]{Spaces and groups with conformal dimension greater than one}
\author[J. M. Mackay]{John M. Mackay}

\address{Department of Mathematics, University of Michigan, Ann Arbor, MI 48109.}
\curraddr{Department of Mathematics, University of Illinois at Urbana-Champaign, Urbana, IL 61801.}
\email{jmmackay@illinois.edu}

\thanks{This research was partially supported by NSF grant DMS-0701515.}
\subjclass[2000]{Primary 51F99; Secondary 20F67, 30C65.}

\usepackage{amsmath,amsthm,amsfonts,graphicx}

\newtheorem{thm}{Theorem}[section]

\newtheorem{cor}[thm]{Corollary}
\newtheorem{lem}[thm]{Lemma}

\theoremstyle{definition}
\newtheorem{defn}[thm]{Definition}

\DeclareMathOperator{\diam}{diam}
\newcommand{\alp}{\alpha}
\newcommand{\bet}{\beta}
\newcommand{\del}{\delta}
\newcommand{\gam}{\gamma}
\newcommand{\eps}{\epsilon}

\newcommand{\lam}{\lambda}
\newcommand{\sig}{\sigma}
\newcommand{\bdry}{\partial_\infty}

\newcommand{\CH}{\mathcal{H}}
\newcommand{\CC}{\mathcal{C}}
\newcommand{\dimC}{\dim_{\mathcal{C}}}

\newcommand{\CM}{\mathcal{M}}

\newcommand{\ra}{\rightarrow}
\newcommand{\RR}{\mathbb{R}}
\newcommand{\NN}{\mathbb{N}}

\begin{document}

\begin{abstract}
We show that if a complete, doubling metric space
is annularly linearly connected then its conformal dimension
is greater than one, quantitatively.
As a consequence, we answer a question of Bonk and Kleiner:
if the boundary of a one-ended hyperbolic group has
no local cut points, then its conformal dimension is
greater than one.
\end{abstract}

\maketitle

\section{Introduction}
\label{sec-intro}

 A standard quasi-symmetry invariant of a metric space $(X,d)$ is 
 its {\em conformal dimension}, introduced by Pansu~\cite{Pan89a}.
 It is defined as the infimal Hausdorff dimension of all 
 metric spaces quasi-symmetric to $X$, denoted here by
 $\dimC(X)$.

 Conformal dimension is a natural concept to consider 
 since in some sense it measures the metric
 dimension of the best shape of $X$; see \cite{BK05a} for
 discussion and references for this kind of uniformization
 problem.  A key application of the definition (and its original motivation)
 is in the study of the conformal structure of the boundary
 at infinity of a negatively curved space.
 
 Besides the trivial bound given by the topological dimension
 of a metric space, the conformal dimension is often difficult to estimate
 from below.
 In this paper we give such a bound for an interesting class
 of metric spaces.

\begin{thm}\label{thm-main}
 Suppose $(X,d)$ is a complete metric space which is doubling and
 annularly linearly connected.
 Then the conformal dimension
 $\dimC(X)$ is at least $C>1$, where
 $C$ depends only on the the constants associated
 to the two conditions above.
\end{thm}

 Recall that a metric space is 
 {\em $N$-doubling} if every ball can be covered using  
 $N$ balls of half the radius.  The annularly linearly
 connected condition is a quantitative analogue of the topological
 conditions of being locally connected and having 
 no local cut points.
 This is made precise in Definition~\ref{defn-annulusllc} and the
 subsequent discussion.
 For now, a good motivating example of
 a space satisfying our hypotheses is the standard square 
 Sierpi\'nski carpet, denoted by $S$.
 
 The original motivation to study such spaces was given by
 a particular application of Theorem~\ref{thm-main}.
 Each Gromov hyperbolic group $G$ has an associated
 boundary at infinity $\bdry G$, a geometric
 object much studied in its relationship to the group structure
 of $G$ (e.g.~\cite{Gro87,Bow98,Kle06}).
 The boundary carries a canonical family of metrics that are
 pairwise quasi-symmetric, and so
 any quasi-symmetry invariant of metric spaces, such as
 conformal dimension, will give
 a quasi-isometry invariant of $G$.
 
 If the boundary of a hyperbolic group is connected and has no local
 cut points, for example if it is 
 homeomorphic to the Sierpi\'nski carpet
 or the Menger curve, then its self-similarity implies that it will satisfy 
 the (a priori stronger) hypotheses of
 Theorem~\ref{thm-main}.
 Thus as a corollary we answer a question
  of Bonk and Kleiner~\cite[Problem 6.2]{BK05a}.

\begin{cor}\label{cor-main}
 Suppose $G$ is a hyperbolic group whose boundary is non-empty, connected and has no local
 cut points.  Then the conformal dimension of $\bdry G$ 
 is strictly greater than one.
\end{cor}

 These topological conditions on the boundary of a hyperbolic group correspond to
 algebraic properties of the group itself.
 The boundary is non-empty when the group is infinite.
 Using Stallings' theorem on the ends of a group, 
 one sees that the boundary is connected if and only if
 the group does not split over a finite group~\cite{Sta68}.
 
 More recently, work of Bowditch \cite[Theorem 6.2]{Bow98}
 shows that if $\bdry G$ is connected and not homeomorphic to $S^1$, then 
 $G$ splits over a virtually cyclic subgroup if and only if $\bdry G$
 has a local cut point.
 
 Using these results, we note a more algebraic version of Corollary~\ref{cor-main}.

\begin{cor}\label{cor-main2}
 Suppose $G$ is a one-ended hyperbolic group whose boundary has conformal dimension one.
 Then either the boundary of $G$ is homeomorphic to $S^1$ (and hence $G$ is virtually Fuchsian),
 or $G$ splits over a virtually cyclic subgroup.
\end{cor}
  
\subsection*{Outline of proof} 
 Let us return to the example of the standard Sierpi\'nski carpet, $S$.
 Since $S$ has topological dimension equal to one, we need to rely on the metric
 structure of $S$ to prove Theorem~\ref{thm-main}.  It is clear that $S$ contains
 an isometrically embedded copy of $C \times [0,1]$, where $C$ equals the standard
 one third Cantor set.  By a lemma of Pansu,
 the conformal dimension of $C \times [0,1]$ equals the
 Hausdorff dimension of $C$ plus one, and so we have that the conformal dimension of
 $S$ is greater than one; see, for example,~\cite[Example 4.3]{Pan89b}.

 In general, we do not have a product structure to exploit.  Nevertheless,
 we construct a family of arcs in our space $X$ akin to the product
 of an interval and a regular Cantor set (of controlled dimension), and then Pansu's
 lemma completes the proof.
 
Let us consider an example of extending a topological statement
to a quantitative metric analogue.  It is well known that a connected,
locally connected, complete metric space $X$ is arc-wise connected.
Less well known is Tukia's analogous metric result (Theorem~\ref{thm-tukia}): a linearly
connected, doubling and complete metric space is connected by
quasi-arcs.  (See Section~\ref{sec-note} for definitions.)

If we now further assume that $X$ has no local cut points
-- as in the situation of Theorem~\ref{thm-main} --
then a topological
argument shows that the product of a Cantor set and the unit interval 
embeds homeomorphically into $X$.  A weaker statement
is that there exists a collection of arcs $\{J_\sig\}$
in $X$ such that, under the topology induced by the Hausdorff metric, 
the set $\{J_\sig\}$ is a topological Cantor set.

We will show a quantitatively controlled analogue of this weaker statement.
First, let $(\CM(X),d_\CH)$ be the (complete) metric space
 consisting of all closed subsets of $X$ with the Hausdorff metric $d_\CH$.
 For each $\sig>0$, we shall denote by $Z_\sig$ a standard Ahlfors regular Cantor set 
 of Hausdorff dimension $\sig$; this is defined precisely in 
 Section~\ref{sec-lotsofarcs}.
  
\begin{thm}\label{thm-arcs}
 For all $L \geq 1$ and $N \geq 1$, there exist $C \geq 1$,
 $\sig > 0$ and $\lam' \geq 1$ such that
 if $X$ is an $L$-annularly linearly connected, $N$-doubling, complete metric 
 space of diameter at least one, then there exists
 a $C$-bi-Lipschitz embedding of $Z_\sig$ into $\CM(X)$, where
 each point in the image is a $\lam'$-quasi-arc of diameter at least $\frac{1}{C}$.
 Moreover, on the image the Hausdorff metric and minimum distance metric
 are comparable with constant $C$.
\end{thm}
 
So, how do we create such a good collection of arcs?
First, use the topological properties
of the space to split one arc into two arcs and apply Tukia's theorem
 (Theorem~\ref{thm-tukia}) to straighten
these arcs into uniformly local quasi-arcs.  Second, repeat this procedure 
in a controlled way by using the compactness properties of the quasi-arcs and spaces.
This process gives four arcs, then eight, and so on, limiting to a
collection of arcs indexed by a Cantor set.  This process is described in
Sections~\ref{sec-lotsofarcs} and \ref{sec-mainresults}.

 In Section~\ref{sec-mainresults} we use Pansu's lemma to complete the
 proof of Theorem~\ref{thm-main}.
 Corollary~\ref{cor-main} follows from a short dynamical argument
 similar to one given by Bonk and Kleiner in \cite{BK05b}.
 
As a final remark, we emphasize that the work here is to show the existence of
a {\em uniform} lower bound, greater than one, on the Hausdorff dimension
of any quasi-symmetrically equivalent
metric.  Pansu gave examples of hyperbolic groups which do not have this property:
the canonical family of
(quasi-symmetrically equivalent) metrics on the boundary contains metrics whose 
Hausdorff dimension is arbitrarily close to, but not equal to, one.
These groups are the fundamental groups of spaces
obtained by gluing together two closed hyperbolic surfaces
along an embedded geodesic of equal length in each, corresponding
to an amalgamation of the two surface groups along embedded cyclic subgroups.
Of course, the boundaries of such groups contain local cut points.

For more discussion on conformal dimension, we refer the reader
 to the Bonk and Kleiner~\cite{BK05a} and Kleiner~\cite{Kle06}.
 Note that these authors work with the Ahlfors
 regular conformal dimension; since this infimum is taken
 over a more restricted class of spaces, 
 it is bounded below by the conformal dimension, and
 thus our result still applies.

\subsection*{Acknowledgments}
The author gratefully thanks Bruce Kleiner for all his help and advice.
He also thanks the Department of Mathematics
at Yale University for its hospitality 
and Enrico Le Donne for commens on an earlier draft of this article.

\section{Background} \label{sec-note}
\subsection{Quasi-arcs and arc straightening}

Basic analytic definitions and results are contained in~\cite{Hei01}.
Although conformal dimension is defined using quasi-symmetric mappings, we
will primarily use geometric arguments inside metric spaces.  

We will need some notation.  A metric space $(X,d)$
 is said to be {\em $L$-linearly connected} for some $L \geq 1$ if
  for all $x,y \in X$ there
exists a continuum $J \ni x,y$
of diameter less than or equal to $L d(x,y)$.
This is also known as the LLC(1) or BT (bounded turning) condition.
We can actually assume that $J$ is an arc, at the
cost of increasing $L$ by an arbitrarily small amount.

As already mentioned, $X$ is {\em doubling} if there exists a constant $N$ such that every ball
can be covered by at most $N$ balls of half the radius.
Note that a complete, doubling metric space is proper: closed balls are compact.

A key tool in creating the collection of arcs in Theorem~\ref{thm-arcs}
is a result of Tukia that straightens arcs into local quasi-arcs.  Before describing
it we need some language to deal with embedded arcs.  Denote
the sub-arc of an arc $A$ between $x$ and $y$ in $A$ by $A[x,y]$.
We say that an arc $A$ in a doubling and complete metric space is an
{\em $\eps$-local $\lam$-quasi-arc} if
$\mathrm{diam}(A[x,y]) \leq \lam d(x,y)$ for all $x, y \in A$ such that
$d(x,y) \leq \eps$.  If this holds for all $\eps >0$,
then we say $A$ is a {\em $\lam$-quasi-arc}.
The terminology is natural since, by a result
of Tukia and V\"ais\"al\"a \cite{TV80}, such an arc is (locally)
 the image of a quasi-symmetric
embedding of the unit interval.

One non-standard definition will be useful to us:
we say that an arc $B$ {\em $\eps$-follows} an arc $A$ if there
exists a (not necessarily continuous) map $p:B \ra A$ such that 
for all $x,y \in B$, $B[x,y]$ is in the $\eps$-neighborhood of
$A[p(x),p(y)]$; in particular, $p$ displaces points at most $\eps$.

We can now state Tukia's theorem.

\begin{thm}[{\cite[Theorem 1B]{Tuk96}}]\label{thm-tukia}
 Suppose $(X,d)$ is a L-linearly connected,  N-doubling, 
 complete metric space.
 For every arc $A$ in $X$ and every $\eps >0$, there is an arc $J$ in the
 $\eps$-neighborhood of $A$ which $\eps$-follows $A$, has the same endpoints as $A$,
 and is an $\alp\eps$-local $\lam$-quasi-arc,
 where $\lam = \lam(L,N) \geq 1$ and $\alp = \alp(L,N) >0$.
\end{thm}

Tukia's original statement concerned subsets of $\RR^n$.  Bonk and Kleiner
\cite[Proposition 3]{BK05b}
used Assouad's embedding theorem to translate it into this language.
For a shorter proof, see \cite[Theorem 1.1]{Mac07a}.

As mentioned in the introduction, this theorem has the 
following independently interesting corollary:

\begin{cor}[Tukia {\cite[Theorem 1A]{Tuk96}}]\label{cor-tukia}
 Every pair of points in a $L$-linearly connected, $N$-doubling,
  complete metric space is
 connected by a $\lam$-quasi-arc, where $\lam = \lam(L,N) \geq 1$.
\end{cor}

\subsection{Hausdorff distance and Gromov-Hausdorff convergence}

We recall some standard definitions and results (for example, see
\cite[Chapters 7,8]{BBI01}).


Suppose $(X,d)$ is a metric space.  We define
the distance between $x \in X$ and $U \subset X$ as 
\[ d(x,U) = \inf \{ d(x,y) : y \in U\}.\]
The $r$-neighborhood of $U$ is the set $N(U,r) = \{x : d(x,U) < r \}$,
where $r \geq 0$.
The Hausdorff distance $d_\CH$ between $U,V \subset X$ is
\[ d_\CH(U,V) = \inf \{ r  \geq 0 : U \subset N(V,r), V \subset N(U,r) \}.\]

%

We say that a sequence of compact metric spaces $\{X_i\}$, $i \in \NN$, converges to
a metric space $X$ in the Gromov-Hausdorff topology if there exist
 $f_i:X \ra X_i$ and $\eps_i \geq 0$ so that
$f_i$ distorts distance by at most an additive
error of $\eps_i$, $N(f_i(X),\eps_i)$ equals $X_i$, and $\eps_i \ra 0$.
(This is equivalent to the usual definition of
convergence with respect to the Gromov-Hausdorff metric.)

If $X$ is $N$-doubling and $\eps >0$, then $X$ has a finite $\eps$-net
of cardinality at most $C(N,\eps) < \infty$.  Therefore, given any
sequence of $N$-doubling spaces their geometry on scale $\eps$ can
be modelled using uniformly finite sets.  An Arzel\`a-Ascoli argument
gives the following result.  For a proof, see
\cite[Theorem 7.4.15]{BBI01}.

\begin{thm} \label{thm-limgh}
Any sequence of $N$-doubling, complete metric spaces of diameter
at most $D$ has
a subsequence that converges in the Gromov-Hausdorff topology to a
complete metric space of diameter at most $D$.
\end{thm}

An analogous argument gives results when we consider configurations
of subsets inside $X$.  As a simple example,
consider a sequence of pairs $\{(X_i,A_i)\}$, 
where each $A_i$ is a closed subset of $X_i$.

We say that $(X_i,A_i)$ converges to $(X,A)$ in the Gromov-Hausdorff
topology, where $A$ is a closed subset of $X$,
 if, as before, there exist $f_i:X \ra X_i$ and $\eps_i \geq 0$
so that $f_i$ distorts distances by at most $\eps_i$,
$N(f_i(X),\eps_i) = X_i$, and $\eps_i \ra 0$.  However, we now also require
that $d_\CH(f_i(A),A_i) \leq \eps_i$.  At the cost of doubling $\eps_i$,
we can assume that $f_i(A) \subset A_i$.

A slightly modified version of the proof of Theorem~\ref{thm-limgh} gives
the following:

\begin{thm} \label{thm-limgh-configurations}
	Suppose for each $i \in \NN$, 
	$X_i$ is an $N$-doubling metric space of diameter at most $D$,
	and $A_i$ is an arc in $X_i$.
	Then there is a subsequence of the configurations
	$(X_i,A_i)$ that converges in the Gromov-Hausdorff
	topology to a limit $(X,A)$.
	
	Moreover, if each $A_i$ is a $\lambda$-quasi-arc, then
	$A$ will be a $\lambda$-quasi-arc also; in particular,
	$A$ is an arc.
\end{thm}

This last claim follows from an elementary argument using the definitions
of convergence and quasi-arcs.

\section{Unzipping arcs}
\label{sec-lotsofarcs}

Consider a complete, locally connected metric space with 
 no local cut points, that is, no connected open set is disconnected by
 removing a point.  In such a space it is straightforward
 to ``unzip'' a given arc $A$ into two disjoint arcs $J_1$ and $J_2$ lying in
 a specified neighborhood of $A$.  Repeating this procedure
 to get four arcs, then eight, and so on, it is possible, with some care,
  to get a limiting set
 homeomorphic to the product of a Cantor set and the interval.  
 Such a limit set is
 useless for our purposes because there is no control on the
 minimum distance between two unzipped arcs, and so no way to get a lower
 bound on conformal dimension that is greater than one.
 We will use compactness type arguments to overcome this problem.
 
We begin by proving the topological unzipping result.  
 
\begin{lem}\label{lem-topsplit}
 Given an arc $A$ in a complete, locally connected metric space with
 no local cut points, and $\eps > 0$, it is possible to find two
 disjoint arcs $J_1$ and $J_2$ in $N(A,\eps)$ such that the endpoints
 of $J_i$ are $\eps$-close to the endpoints of $A$.  Furthermore,
 the arcs $J_i$ $\eps$-follow the arc $A$.
\end{lem}

\begin{proof}
 Here, $B_0(x,r)$ denotes the connected component of an open
 ball $B(x,r) \subset X$ that contains its center $x$.
 As $X$ is locally connected, $B_0(x,r)$ is always open and connected,
 and, moreover, $B_0(x,r)\setminus \{x\}$ is also open and connected
 because $x$ is not a local cut point.  Any open and connected
 subset of $X$ is arcwise connected.

 Let $a$ and $b$ be the initial
 and final points of $A$ respectively  (in a fixed order given by the topology).
 We are going to define $J_1$ and $J_2$ inductively.
 There exists $w \in B_0(a,\eps) \setminus A$; otherwise,
 there would be a open set in $X$ homeomorphic to an arc segment,
 violating the ``no local cut point'' condition.
 Now join $w$ to $a$ by an arc in $B_0(a, \eps)$.
 Stop this arc at $x$, the first time it meets $A$, and call it 
 $J_1 = J_1[w,x]$.  Set $J_2 = A[a,x]$.
(Perhaps $x=a$, but this is not a problem).
 
 Now we have two head segments for $J_1$ and $J_2$ meeting only at
 $x \in A$, and we want to unzip this configuration further
 along $A$.  This is possible since in $B_0(x, \eps)$ there is 
 a tripod type configuration with two incoming arcs $J_1$ and $J_2$ 
 and one outgoing arc $A[x,b]$.  As noted above, $B_0(x,\eps) \setminus \{x\}$
 is arcwise connected, and so we can find an arc in this set that joins
 some point in $J_1$ (not $x$) to a point in $A[x,b]$ (not $x$).
 The arc may meet $J_1$, $J_2$ and $A[x,b]$ in many places
 but there must be some sub-arc $A'$ joining some point in $J_1$ or $J_2$ to
 some point $y$ in $A$ with interior disjoint from them all.  
 (See Figure~\ref{fig-detour1}, where $A'$ is emphasized.) Use $A'$
 to detour one of $J_1$ and $J_2$ around $x$ to the new unzipping point $y$,
 and extend the other $J_i$ to $y$ using $A[x,y]$.
 
 \begin{figure}
  \begin{center}
  \includegraphics[width=0.7\textwidth]{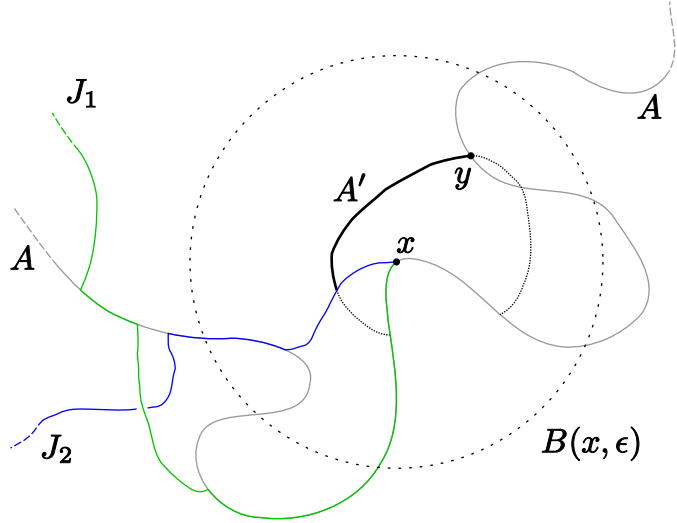}
  \end{center}
  \caption{Unzipping an arc}\label{fig-detour1}
 \end{figure}
 
 What if this unzipping process approaches a limit before we are
 $\eps$-close to the final point $b$ in $A$?  This cannot happen.
 Suppose it is not possible to unzip past $z \in A$.
 Since $B_0(z,\frac{\eps}{4}) \setminus \{z\}$ is arcwise connected,
 inside this set we can construct an arc $A''$ that detours around $z$, from 
 $z_1 \in A$ to $z_2 \in A$, where $z_1 < z < z_2$ in the order on $A$.
 
 Now by the limit point hypothesis, we can unzip $J_1$ and $J_2$ 
 past $z_1$ to $x$, where $z_1 < x < z$.
 To continue the construction of $J_1$ and $J_2$ past $z$, 
 find the arc given by following $z_2$ to $z_1$
 along $A''$, stopping if one of $J_1$ or $J_2$ is met.  If we reach
 $z_1$ without intersecting $J_1$ or $J_2$, as is the case in 
 Figure~\ref{fig-detour2}, then continue to follow $A$ from $z_1$ towards $z$.
 By the construction of $J_1$ and $J_2$, this
 arc will meet $J_1$ or $J_2$ before reaching $z$.
 In either case, this arc can be used as a 
 legitimate detour around $x$ and $z$, contradicting the assumption on $z$.
 Thus it is possible to continue unzipping until $x \in B(b, \frac{\eps}{2})$.
 
 \begin{figure}
  \begin{center}
  \includegraphics[width=0.7\textwidth]{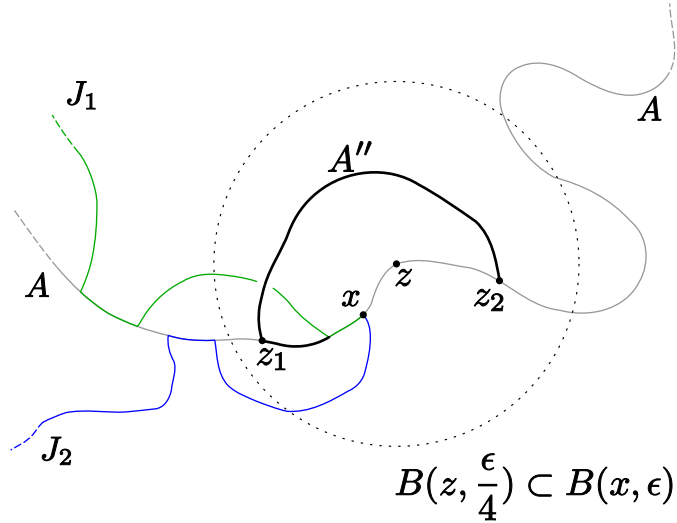}
  \end{center}
  \caption{Avoiding a limit point}\label{fig-detour2}
 \end{figure}
 
 It remains to find labellings $f_i:J_i\ra A$, for $i=1,2$.
 Define $f_i$ to be the identity on $J_i \cap A$.  Each element
 $v$ of $J_i \setminus A$ was created to detour around some point $x \in A$;
 define $f_i(v)$ to equal $x$.
 This labelling coarsely preserves order as desired.
\end{proof}

We would like to give a lower bound for the distance between the two
 split arcs.  To do this we need a quantitative metric version of
 being locally connected with no local cut points.
Let $A(p,r,R)$ be the annulus $\overline{B}(p,R) \setminus B(p,r)$.
 
\begin{defn}\label{defn-annulusllc}
 We say a metric space $X$ is {\em ($L$-)annularly linearly connected} for some
  $L \geq 1$ if whenever $p \in X$, and $x,y \in A(p,r,2r)$ for some $r>0$,
  there exists an arc $J$ joining $x$ to
  $y$ that lies in the annulus $A(p,\frac{r}{L},2L r)$.  Furthermore,
  we assume that $X$ is connected and complete.
\end{defn}

At the cost of replacing $L$ by $8L$, we may assume that such a space is
also $L$-linearly connected.


This condition is stronger than the usual LLC (linearly locally connected)
condition~\cite[Definition 3.12]{HK98}, and is mentioned in~\cite[Remark 3.19]{HK98}.
It is called LLC (linearly locally convex) in~\cite[Section 2]{BMS01}; the authors
of this paper use this condition in the context of spaces that satisfy a Poincar\'e
inequality.
 
The key feature of Definition~\ref{defn-annulusllc} is that, unlike the usual LLC
condition, it preserved under Gromov-Hausdorff convergence.
To be precise,
if $\{X_i\}$ is a sequence of $L$-annularly linearly connected, uniformly doubling, complete metric spaces
 and $X_i \ra X_\infty$ in the Gromov-Hausdorff topology, then $X_\infty$
 is $L'$-annularly linearly connected for any $L'>L$.  (We need to increase $L$ slightly
 to allow ourselves to connect by arcs rather than just continua.)
 Furthermore, annularly linearly connected implies that there are no local cut points.

As a side remark, let us note that 
we do need a stronger condition than no local cut points as a hypothesis for
 Theorem~\ref{thm-main}: it is straightforward to modify the
 Sierpi\'nski carpet construction to get a doubling, linearly connected, complete
 metric space with no local cut points whose {\em Hausdorff} dimension is one.
 Therefore, its conformal dimension is also one.
 
Now for the remainder of this section we will assume that $L$ and $N$ are fixed
 constants, and $\lam \geq 1$, $\alp \in (0,1]$ are
 as given by Theorem~\ref{thm-tukia}.
 Consider the collection $\CC$ of all
 $\lam$-quasi-arcs $A$ in any complete metric space $X$ that is
 $L$-annularly linearly connected and $N$-doubling, and whose endpoints $a$ and $b$ satisfy
 $d(a,b) \in [\frac{1}{R},R]$ for some $R \geq 1$.
 Fix $\eps > 0$, and consider the supremum of 
 possible separations of two arcs split from $A$ by
 the topological lemma above.  Call this $\del_A$ ($\del_A > 0$).
 
\begin{lem}\label{lem-metsplit1}
 There exists $\del^\star = \del^\star(\lam, L, N, \eps, R) > 0$
 such that
 for all $A \in \CC$, $\del_A > \del^\star$.
\end{lem}

\begin{proof}
 If not, then we can find a sequence of arcs $A_i \subset X_i$ such that
 $\del_{A_i} < \frac{1}{i}$.  Let $a_i$ and $b_i$ denote the endpoints of $A_i$.
 We are only interested
 in what happens inside the ball $B_i := B(a_i, 10L(\lam R+\eps))$.
 As the sequence of configurations $(B_i, A_i, a_i, b_i)$ is
 precompact in the Gromov-Hausdorff topology, by an argument
 similar to Theorem~\ref{thm-limgh-configurations} we can take 
 a subsequence converging to $(B_\infty, A_\infty, a_\infty, b_\infty)$,
 where $A_\infty$ is a $\lam$-quasi-arc inside $B_\infty$ with endpoints
 $a_\infty$ and $b_\infty$.
 
 Convergence here means that there exist constants $C_i \ra 0$ and maps
 $f_i : B_\infty \ra B_i$ such that $f_i$ distorts distances by
 an additive error of at most $C_i$, and every
 point of $B_i$ is within $C_i$ of $f_i(B_\infty)$.  Furthermore,
 $f_i(A_\infty) \subset A_i$, $f_i(a_\infty)=a_i$ and $f_i(b_\infty) = b_i$.

 Since $B_\infty$ will be $L$-annularly linearly connected (away from the edge of the ball),
 it will have no local cut points in its interior.
 Consequently, we can split $A_\infty$ into two arcs $J_1$ and $J_2$ using
 Lemma~\ref{lem-topsplit} inside an $\frac{\eps}{3}$-neighborhood 
 of $A_\infty$.  These arcs are disjoint so they are separated by some
 distance $0 < \del' \leq \frac{\eps}{3}$.
 The remainder of the proof consists of showing that this contradicts
 the assumption on $A_i \subset B_i$ for some large $i$.

 For sufficiently large $i$,
 $C_i \leq \frac{\del'}{8L}$ because $C_i \ra 0$ as $i \ra \infty$.
 For $j=1,\,2$, the arc $J_j$ in $B_\infty$ contains a discrete path $D_j$ with $C_i$-sized
 jumps that corresponds to a discrete path $D_j' = f_i(D_j)$ in $X_i$ with
 $2C_i \leq \frac{\del'}{4L}$ jumps.
 The $L$-linearly connected condition can then be used
 to join each $D_j'$ up into a continuous arc $J_j'$.

 To be precise, if $D_j' = \{p_1,\ldots,p_M\}$, join $p_1$ to $p_2$
 by an arc $J_j'$ of diameter at most $2C_i L \leq \frac{\del'}{4}$.  
 Assume that, at a stage $k$, we have
 an arc $J_j'$ from $p_1$ to $p_k$.  There is an arc $I$ of
 diameter at most $\frac{\del'}{4}$ joining $p_{k+1}$ to $p_k$.
 We extend $J_j'$ to $p_{k+1}$ by following $I$ from $p_{k+1}$ to $p_k$,
 stopping at $x$, the first time it meets $J_j'$, 
 and gluing together $J_j'[p_1,x]$ and $I[x,p_{k+1}]$ to make a new arc $J_j'$,
 and we repeat this until $k = M$.
 Define a map $h_j:J_j' \ra D_j'$ that sends each of the points added at stage $k$
 to the point $p_k$.  Note that for all $x,\,y \in J_j'$, $J_j'[x,y] \subset
 N(D_j'[h_j(x),h_j(y)],\frac{\del'}{4})$; in a coarse sense,
 $J_j'$ $\frac{\del'}{4}$-follows $D_j'$.

 By construction, $J_1'$ and $J_2'$ are $\frac{\del'}{4}$-separated and
 $\eps$-close to $A_i$, but to get a contradiction we need them to
 $\eps$-follow $A_i$.

 Since $A_\infty$ and $A_i$ are both $\lam$-quasi-arcs,
 Lemma~\ref{fact-qa-approx} below implies that
 for all $x$, $y \in A_\infty$, 
 $f_i(A_\infty[x,y])$ is contained in the $((2C_i\lam+C_i)\lam+C_i)$-neighborhood
 of $A_i[f_i(x),f_i(y)]$.
 For each $j$, we can lift the map $h_j:J_j' \ra D_j'$ to
 a map $h_j':J_j' \ra D_j \subset B_\infty$.
 By Lemma~\ref{lem-topsplit}, $D_j$ $\frac{\eps}{3}$-follows $A_\infty$,
 so further compose with the associated map $D_j \ra A_\infty$.
 Finally, compose with $f_i:A_\infty \ra A_i$.

 The composed maps $J_j' \ra D_j \ra A_\infty \ra A_i$, for
 each $j$, show that each
 $J_j'$ follows $A_i$ with constant 
 $\left( \frac{\del'}{4}+C_i+\frac{\eps}{3}+
   (2C_i\lam+C_i)\lam+C_i \right)$.
 This is smaller than $\eps$ for sufficiently large $i$ 
 because $C_i \ra 0$ as $i \ra \infty$.
 We have contradicted our initial assumption, so the proof is complete.
\end{proof}

We used the following lemma in the proof:
\begin{lem}\label{fact-qa-approx}
If $A$ and $A'$ are $\lam$-quasi-arcs, and $f:A \ra A'$ is a map distorting
distances by at most $C$, then for all $x$ and $y$ in $A$,
\[
f(A[x,y]) \subset N( A'[f(x),f(y)], (2C\lam + C)\lam +C).
\]
\end{lem}
\begin{proof}
 Let $x = p_0 < p_1 < \cdots < p_n = y$
be a chain of points in $A$ so that the diameter of $A[p_{i-1}, p_i]$ is
less than $C$, for $i = 1, \ldots, n$.

Let $x' = f(x)$, $y' = f(y)$, and $p_i' = f(p_i)$.  Order $A'$ so that
$x' \leq y'$.  Let $l \geq 1$ be the greatest index so that
$p_{l}' \leq x' \leq p_{l+1}'$.  Let $m$, $l \leq m \leq n$, be the
smallest index so that $p_m' \leq y' \leq p_{m+1}'$.
So $p_{l+1}', \ldots, p_m' \in A'[x',y']$.

Since $d(p_i', p_{i+1}') \leq 2C$, we have $d(p_{l+1}', x')$
and $d(p_m', y')$ are both less than or equal to $2C\lam$.
This lifts, by $f$, to give that $d(p_l, x)$ and $d(p_m, y)$ are both
less than or equal to $2C\lam + C$, and so
\[
\diam(A[x,p_{l+1}]) \leq (2C\lam+C)\lam \quad \text{and} \quad
\diam(A[p_m, y]) \leq (2C\lam+C)\lam.
\]

Therefore,
\begin{align*}
 f(A[x,y]) & \subset f(A[x,p_{l+1}] \cup A[p_{l+1}, p_m] \cup A[p_m, y]) \\
 & \subset N(\{x',y'\},(2C\lam+C)\lam+C) 
      \cup N(A'[x',y'], 2C\lam) \\
 & \subset N(A'[f(x),f(y)], (2C\lam+C)\lam+C ). \qedhere
\end{align*}
\end{proof}  

 The important point to note in Lemma~\ref{lem-metsplit1}
  was the presence of the diameter
 constraint $R$ allowing us to use a compactness type technique.  Without
 this constraint we have various problems: our sequence of
 counterexamples still converges in some
 sense, but could give an
 unbounded arc.  Topological unzipping still works but the resulting arcs
 would not necessarily have a positive lower bound on
 separation.
 
 We can deal with the problem of no diameter bounds
 by dividing the problem into two collections of non-interacting smaller
  problems.   To be precise, given a $\lam$-quasi-arc $A$, or even just a
 local $\lam$-quasi-arc, we can use Lemma~\ref{lem-metsplit1} on
 uniformly spaced out 
 small subarcs of $A$ (that are genuine $\lam$-quasi-arcs)
 with a sufficiently small $\eps$ value -- this is the first 
 collection of problems.
 
 Now the second collection of independent problems is how to join
 together two of these small splittings with two disjoint
 arcs having uniform bound on their separation -- but this a problem
 with bounded diameter!  So compactness arguments allow us to fix this
 and to remove
 the dependence of $\del^\star$ on $R$ in Lemma~\ref{lem-metsplit1}.

\begin{lem} \label{lem-metsplit2}
 Given $0 < \eps \leq \mathrm{diam}(X)$ and an $\alp\eps$-local
 $\lam$-quasi-arc $A$ in $X$, where $\alp \in (0,1]$ is a constant,
 there exists $\del^\star = \del^\star(\lam, L, N,\alp) > 0$
 such that for all $\del<\del^\star$ we can split $A$ into
 two arcs that $\eps$-follow $A$ and that are $\del\eps$-separated.
\end{lem}
\begin{proof}
 Without loss of generality we can rescale to $\eps = 1$.
 As before, choose a linear order on $A$ compatible with its topology.
 Let $x_0$ be the first point in $A$, and $y_0$ be the first point
 at distance $D_1 = \frac{\alp}{5\lam}$ from $x_0$.  
 (If there is no such point, $\mathrm{diam}(A) \leq \frac{1}{5} = \frac{\eps}{5}$ 
  and so we can split $A$ into two points that are $\frac{\eps}{2}$-separated.)
 Label the next point at distance $D_1$ from $y_0$ by $x_1$.  Continue
 in this manner with all jumps $D_1$ until the last label $y_n$,
 with $d(x_n,y_n) \in [D_1,3D_1)$.
 
 Let $D_2 = \frac{1}{4}D_1 = \frac{\alp}{20\lam}$, and
 $D_3 = \frac{1}{2\lam(L\lam+2)}D_2$.  We can control
 the interactions of the collection of sub-arcs of types
 $A[x_i,y_i]$ and $A[y_i, x_{i+1}]$: the $D_3$ neighborhoods of
 two different such sub-arcs are disjoint outside the collection
 of balls $\{B(x_i,D_2)\} \cup \{B(y_i,D_2)\}$.  This is because
 otherwise there are points $z$ and $z'$ in two different sub-arcs
 that satisfy $d(z,z') \leq 2D_3 < \alp$, so the diameter of 
 $A[z,z']$ is less than $2\lam D_3 < \frac{1}{2}D_2$ -- but
 $A[z,z']$ has to pass through the center of a $D_2$-ball that
 does not contain $z$ or $z'$, which is a contradiction.
 
 Now $A[x_i,y_i]$ is a $\lam$-quasi-arc, and we 
 use Lemma~\ref{lem-metsplit1} to create 
 $J_i$ and $J_i'$ in a $\frac{1}{2}D_3$ neighborhood
 of $A[x_i,y_i]$ that are $\frac{1}{2}\del_0$-separated for some
 $\del_0=\del_0(\lam,L,N,D_3)>0$.
 By applying Theorem~\ref{thm-tukia} to straighten the arcs,
 we may assume that they are $\lam'$-quasi-arcs in a $D_3$
 neighborhood of $A[x_i,y_i]$
 that are $\frac{1}{4}\del_0$-separated,
 where $\lam' = \lam'(L,N,\del_0,D_3)$.
 
 We want to join up the pair of arcs $J_i$ and $J_i'$ ending
 in $B(y_i, D_3)$ to the arcs $J_{i+1}$ and $J_{i+1}'$ starting
 in $B(x_{i+1}, D_3)$, without altering the setup outside the
 set $\mathrm{Join}(i) = B(y_i, D_2) \cup N(A[y_i,x_{i+1}],D_3) \cup B(x_{i+1}, D_2)$.
 Figure~\ref{fig-joining} shows this configuration.
 We will do this joining in two stages: first, a topological
 joining that keeps the arcs disjoint,
 and second, a quantitative version that controls the
 separation of the arcs in the joining.
 
 \begin{figure}
  \begin{center}
  \includegraphics[width=0.9\textwidth]{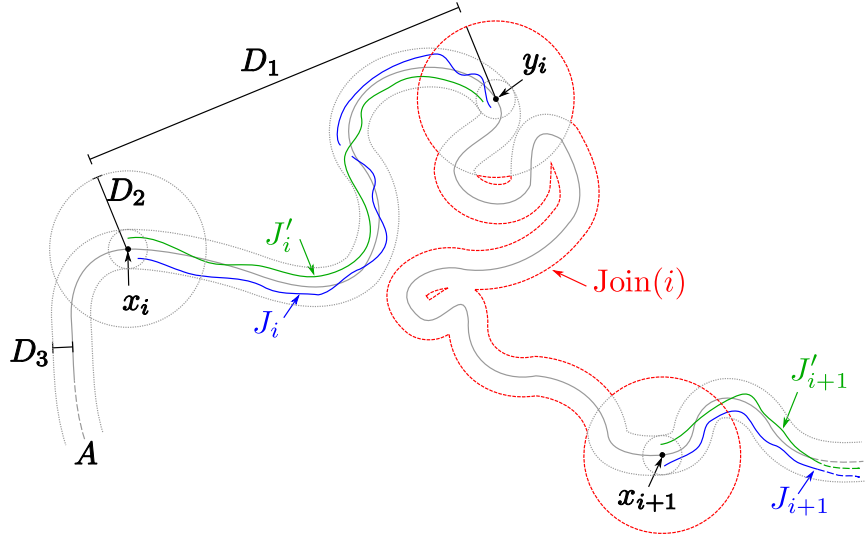}
  \end{center}
  \caption{Joining unzipped arcs}\label{fig-joining}
 \end{figure}

{\em Topological joining:} Join the endpoints of $J_i$ and $J_i'$ to the
 arc $A$ in the ball $B(y_i, L D_3)$ and the endpoints of 
 $J_{i+1}$ and $J_{i+1}'$ to $A$ in the ball $B(x_{i+1}, L D_3)$.
 Use the topological unzipping argument of Lemma~\ref{lem-topsplit}
 to unzip $A$ along
 this segment resulting in `wiring' the pair $(J_i, J_i')$ to the
 pair $(J_{i+1}, J_{i+1}')$ (not necessarily in that order)
 inside $\mathrm{Join}(i)$.  These arcs are disjoint, and so separated
 by some distance $\del>0$.

{\em Quantitative bound on $\del$:} If there is no quantitative
 lower bound on $\del$ then there are configurations (relabeling
 for convenience our joining arcs)
 \[ \CC^n = (X^n, A^n, J_1^n, J_1'^n, J_2^n, J_2'^n), \]
 where the best joining of the pair $J_1^n$ and $J_1'^n$ to the
 pair $J_2^n$, $J_2'^n$ is at most $\frac{1}{n}$-separated.

 But this configuration is precompact in the Gromov-Hausdorff
 topology as the $X^n$ are all $N$-doubling, and the arcs are all
 uniform quasi-arcs.  (This is the importance of Tukia's
 theorem.)  So we can take a subsequence converging to a configuration
 \[ \CC^\infty = (X^\infty, A^\infty, J_1^\infty, J_1'^\infty,
    J_2^\infty, J_2'^\infty) \] in a suitable ball,
 and join the arcs using the topological
 method above, giving some valid rewiring with some positive separation
 $\del^\infty > 0$.  Following the proof of Lemma~\ref{lem-metsplit1}
 we can lift this to $\CC^n$ for sufficiently large $n$
 retaining a separation of $\frac{1}{2}\del^\infty > 0$, which is a contradiction
 for large $n$.
 
 Now since we have some $\del^\star > 0$ to use when joining
 together our wirings in the disjoint collection of all $\mathrm{Join}(i)$,
 we can apply this procedure for all $i$ to create two
 arcs along $A$ that are $\del^\star$-separated, for $\del^\star$
 depending only on $\lam$, $L$, $N$, and $\alp$ as
 desired.  We assumed $\eps = 1$, but rescaling to any $\eps$
 gives the same
 conclusion with our resulting arcs $\del^\star\eps$-separated.
\end{proof}

 
\section{Bounding the conformal dimension from below}
\label{sec-mainresults}

 We now can use the unzipping results of Section~\ref{sec-lotsofarcs}
 (Lemma~\ref{lem-metsplit2}) to create a Cantor set of arcs.
 
 By a Cantor set we mean the space $Z = \{0,1\}^{\NN}$ with an
 \mbox{(ultra-)metric} 
 \[ d_{\sig}((a_1,a_2,\ldots),(b_1,b_2,\ldots)) = 
       \exp(-(\log(2)/\sig) \inf\{n|a_n \neq b_n\}), \]
 where $\sig>0$ is a constant.  (Recall that, by convention, the infimum of the
 empty set is positive infinity.)  The space $(Z, d_{\sig})$ has Hausdorff
 dimension $\sig$, and is Ahlfors regular since there is a Borel probability 
 measure $\nu_\sig $ on $Z$ that satisfies
 $ \frac{1}{2} r^\sig \leq \nu_\sig(B(z,r)) 
   \leq r^\sig$,
 for all $z \in Z$ and $r \leq \mathrm{diam}(Z)$.
 
 Returning to the metric space $(X,d)$
 of Theorem~\ref{thm-main}, we can now prove Theorem~\ref{thm-arcs}.

\begin{proof}[Proof of Theorem~\ref{thm-arcs}]
 Begin with any arc $J'$, assume it has endpoints one unit apart and
 apply Theorem~\ref{thm-tukia} to $J'$ and $\eps = \frac{1}{10}$
 to get $J_{\emptyset}$, a $\lam$-quasi-arc on scales below $\frac{\alp}{10}$.
 Let our scaling factor be 
  $\bet = \frac{\alp\del^\star}{32\lam} \leq \frac{1}{32}$.
 
 We can assume that, for a given $n$, we have a collection of
 $\lam$-quasi-arcs on scales below $\bet^n$, written as
 $\{ J_{a_1 a_2 \ldots a_n} | a_i \in \{0,1\}, 1 \leq i \leq n\}$,
 and that these arcs are $\bet^n$ separated.
 
 For each $J_{a_1 a_2 \ldots a_n}$, we split it into two arcs
 using Lemma~\ref{lem-metsplit2}
 applied to $\eps = \frac{1}{8} \bet^n$,
 then straighten each
 arc using Theorem~\ref{thm-tukia} with 
 $\eps = \frac{\del^\star}{32} \bet^{n}$
 to get two new arcs $J_{a_1 a_2 \ldots a_n 0}$ and 
 $J_{a_1 a_2 \ldots a_n 1}$ that are $\lam$-quasi-arcs on scales
 below 
 $\frac{\alp\del^\star}{32} \bet^n \geq \bet^{n+1}$,
 and are $\frac{\del^\star}{16} \bet^n \geq \bet^{n+1}$
 separated.  In fact, all the arcs created at this stage
 are $\bet^{n+1}$ separated as the new arcs arising from different
 arcs in the previous generation can only get 
 $2\left( \frac{1}{8}\bet^n + \frac{\del^\star}{32}\bet^n \right)
     < \frac{1}{2}\bet^n$
   closer,
 thus remaining at least $\bet^{n+1}$ apart.
 
 At this point it is useful to record the following.
\begin{lem}\label{fact-qa}
 If $J$ is a $\lam$-quasi-arc on scales below $\eps$, and
 we have an arc $J' \subset N(J, \frac{\eps}{4})$, whose endpoints
 are $\frac{\eps}{4}$ close to those of $J$, then we must have
 $J \subset N(J', \lam\eps)$.  In particular,
 $d_\CH(J,J') \leq \lam\eps$.
\end{lem}
 

 Given a sequence $a = (a_1, a_2, \ldots) \in \{0,1\}^\NN$,
 the sequence of arcs $J_{\emptyset},$ $J_{a_1},$ $J_{a_1 a_2}, \ldots$
 is Cauchy in the Hausdorff metric (using Lemma~\ref{fact-qa}),
 and hence convergent to
 $J_{a_1 a_2 \ldots}=J_a$, a set of diameter at least $\frac{1}{2}$.
 A priori, this set need not be an
 arc, but only compact and connected.  (This is actually enough
 to apply the argument of Pansu's lemma.) 
 However, for each $n$ we know that
 $J_{a_1 a_2 \ldots a_n}$ is a $\bet^n$-local $\lam$-quasi-arc
 that $\bet^n$-follows $J_{a_1 a_2 \ldots a_{n-1}}$, and we know
 that $\bet < \min\big\{ \frac{1}{4+2\lam}, \frac{1}{10}\big\}$.  Using these facts,
 \cite[Lemma~2.2]{Mac07a}
 shows that $J_a$ is a $\lam'$-quasi-arc, with 
 $\lam'=\lam'(\bet,L,N)=\lam'(L,N)$, that $\bet^n$-follows
 $J_{a_1 a_2 \ldots a_n}$ for each $n$.

 (Finding quasi-arcs in the limit is not unexpected since on
 each scale the limit set will look like the quasi-arc 
 approximation on the same scale.)
 
 If we set $\CM(X)$ to be the set of all closed sets in $X$,
 we can define a map $F: Z \ra \CM(X)$ by $F(a) = J_a$.
 Let $\mathcal{J} = F(Z)$ be the image of this map and 
 choose the metric $d_{\sig}$ for $Z$, 
  $\sig = \frac{-\log(2)}{\log(\bet)} > 0$.
 It remains to show that $F:(Z,d_\sig) \ra (\CM(X),d_\CH)$
 is a bi-Lipschitz embedding.
 
 Take $a = (a_1, a_2, \ldots), b = (b_1, b_2, \ldots) \in Z$.  Then
 $d_\sig (a, b) \in (\bet^{n+1}, \bet^n]$ if and only if
 $a_i = b_i$ for $1 \leq i < n$ and $a_n \neq b_n$.
 By construction, and a geometric series, 
 $J_a \subset N(J_{a_1 \ldots a_n},\frac{1}{4}\bet^n)$, and so
 as $n$ stage arcs are $\bet^n$ separated,
 we have
\begin{equation}\label{eqn-d-low-est}
 d_\CH(J_a,J_b) \geq d(J_a,J_b) \geq \frac{1}{2}\bet^n 
      \geq \frac{1}{2} d_\sig(a,b).
\end{equation}

 Conversely, applying the triangle inequality and
 Lemma~\ref{fact-qa}, we have
\begin{equation}\label{eqn-d-upp-est}
 d_\CH(J_a,J_b) \leq d_\CH(J_a, J_{a_1 \ldots a_{n-1}})
  + d_\CH(J_{b_1 \ldots b_{n-1}}, J_b)
  \leq 2 \lam \bet^{n-1} \leq \frac{2\lam}{\bet^2} d_\sig(a,b),
\end{equation}
so $F$ is bi-Lipschitz, quantitatively.

 As a final remark, note that there is a natural measure
 $\mu_\sig = F_\ast(\nu_\sig)$ on $\mathcal{J}$.
 The estimates \eqref{eqn-d-low-est} and \eqref{eqn-d-upp-est}
 imply that, for any ball $B(x, r) \subset X$,
 the set $\{ J_a \in \mathcal{J} | J_a \cap B(x,r) \neq \emptyset \}$
 is measurable (in fact open), and if two arcs $J_a$ and $J_b$ both
 meet this ball, we have
 $2r \geq d(J_a, J_b) \geq \frac{1}{2} d_\sig(a,b)$, and so
 \[ \mu_\sig \{ J_a \in \mathcal{J} | J_a \cap B(x,r) \neq \emptyset \}
   \leq 4^\sig r^\sig. \qedhere \]
\end{proof}

We now prove our main theorem.

\begin{proof}[Proof of Theorem~\ref{thm-main}]
The construction of Theorem~\ref{thm-arcs} gives a lower bound for conformal dimension by virtue of
the following lemma of Pansu \cite[Lemma 6.3]{Pan89b}.
  This version is due to Bourdon.

\begin{lem}[{\cite[Lemma 1.6]{Bou95}}]\label{lem-pansu}
 Suppose that $(X,d)$ is a uniformly perfect,
 compact metric space containing a collection
 of arcs $\CC = \{ \gam_i | i \in I\}$ whose diameters 
 are bounded away from zero.  
 Suppose further that we have a Borel probability measure $\mu$
 on $\CC$ and
 constants $A>0$, $\sig \geq 0$ such that, for all balls $B(x,r)$ in $X$, the
 set $\{ \gam \in \CC | \gam \cap B(x,r) \neq \emptyset \}$
 is $\mu$-measurable with measure at most $A r^\sig$.
 Then the conformal dimension of $X$ is at least $1+\frac{\sig}{\tau-\sig}$,
 where $\tau$ is the packing dimension of $X$, and in fact $\tau-\sig \geq 1$.
\end{lem}

In our case $X$ may be non-compact, but it is proper and all arcs 
 $\gam \in \CC$ lie in some fixed (compact) ball in $X$.
 The packing dimension of $X$ is finite and bounded from above by
 a constant derived from the doubling constant $N$.
 Furthermore, $X$ is connected, so it is certainly uniformly perfect.
 
 Following Theorem~\ref{thm-arcs},
 we apply Lemma~\ref{lem-pansu} with 
 $\CC = \mathcal{J}$, $\mu = \mu_\sig$ and $A = 4^\sig$,
 where $\sig$ depends only on $L$ and $N$,
 to find a lower bound for the conformal
 dimension of $C=C(L,N) > 1$.
\end{proof}

We now apply our theorem to the case of conformal boundaries of hyperbolic
groups.

\begin{proof}[Proof of Corollary~\ref{cor-main}]
 In \cite[Proposition 4]{BK05b}, Bonk and Kleiner show that $\bdry G$
 with some visual metric $d$ is compact, doubling and linearly connected.  It remains
 to show that $(X,d) = (\bdry G,d)$ is annularly linearly connected,
  but this follows by a proof similar to that of Bonk and Kleiner's proposition.
 
 Suppose $(X,d)$ is not annularly linearly connected.
 Then there is a sequence of annuli $A_n = A(z_n, r_n, 2r_n)$
 containing points $x_n$ and $y_n$ such that there is no arc joining
 $x_n$ to $y_n$ inside $A(z_n, \frac{1}{n}r_n, 2n r_n)$.
 As $X$ is compact we have $r_n \ra 0$; otherwise, there would
 be a subsequence  $n_j \ra \infty$ as $j \ra \infty$
 with $r_{n_j} > \eps >0$ for some $\eps$.  In this case, take
 further subsequences so that $r_{n_j} \ra r_\infty \in
    [\eps, \mathrm{diam}(X)]$,
 $z_{n_j} \ra z_\infty$, $x_{n_j} \ra x_\infty$, and
 $y_{n_j} \ra y_\infty$.  Then a contradiction follows from
 the fact that $z_\infty$ is not a local cut point, so 
 we must have $r_n \ra 0$.
 
 Now we can consider the rescaled sequence $(X, \frac{1}{r_n}d, z_n)$.
 By doubling, this subconverges to a limit $(W, d_W, z_\infty)$ with
 respect to pointed Gromov-Hausdorff convergence.  By
 \cite[Lemma 5.2]{BK02b}, $W$ is homeomorphic
 to $\bdry G \setminus \{p\}$ for some $p$, and so
 $z_\infty$ cannot be a local cut point in $W$.  So we can connect
 the components of $A(z_\infty, 0.9, 2.1)$ in $W \setminus {z_\infty}$
 by finitely many compact sets, and these must lie in some
 $A(z_\infty, 1/M, 2M)$ for $1\leq M < \infty$.  For sufficiently
 large $n$ we can lift these connecting sets to
 $A(z_n, \frac{1}{2M}r_n, 4M r_n)$, contradicting our hypothesis.
 
 In conclusion, $\bdry G$ is annularly linearly connected,
  doubling and complete, 
 and so Theorem~\ref{thm-main} gives that the conformal dimension
 of $\bdry G$ is strictly greater than one.
\end{proof}

\bibliographystyle{amsalpha}
\bibliography{confdim}

\end{document}